\documentclass{article}
\usepackage{amsmath,amssymb,amsthm,mathtools}
\usepackage{thm-restate}
\usepackage[margin=1in]{geometry}

\newtheorem{theorem}{Theorem}[section]
\newtheorem{proposition}[theorem]{Proposition}
\newtheorem{corollary}[theorem]{Corollary}
\newtheorem{lemma}[theorem]{Lemma}
\newtheorem{remark}[theorem]{Remark}
\newtheorem{definition}[theorem]{Definition}
\newtheorem{example}[theorem]{Example}

\newcommand{\C}{\mathbb C}
\newcommand{\E}{\mathbb E}
\newcommand{\DN}{\mathcal D_N}
\newcommand{\St}{\mathrm{St}}
\newcommand{\op}{\mathrm{op}}

\newcommand{\ip}[2]{\left\langle #1,#2\right\rangle}
\newcommand{\Arow}{\left\|\sum_{i=1}^n A_i^2\right\|_{\op}^{1/2}}

\DeclareMathOperator{\Tr}{Tr}
\DeclareMathOperator{\Rea}{Re}
\DeclareMathOperator{\Ima}{Im}
\DeclareMathOperator{\vecop}{vec}

\title{A Sudakov--Fernique proof of Lehner-type edge bounds for matrix-valued GUE sums}
\author{Beno\^{i}t Collins\thanks{Department of Mathematics, Graduate School of Science,
  Kyoto University, Kyoto 606-8502, Japan.
  E-mail: \texttt{collins@math.kyoto-u.ac.jp}}
\and
Yuta Yamagishi\thanks{Department of Mathematics, Graduate School of Science,
  Kyoto University, Kyoto 606-8502, Japan.
  E-mail: \texttt{yamagishi.yuta.48r@st.kyoto-u.ac.jp}}
  }
\date{\today}

\begin{document}

\maketitle

\begin{abstract}
	Let $A_0,A_1,\ldots,A_n\in M_N(\C)$ be Hermitian matrices, and let
	$G_1,\ldots,G_n$ be independent $M\times M$ GUE matrices,
	normalized so that $\|M^{-1/2}G_i\|\to2$ almost surely as $M\to\infty$.  We
	study the spectral edges and operator norm of
	\[
		H_M=A_0\otimes I_M+\frac1{\sqrt M}\sum_{i=1}^n A_i\otimes G_i .
	\]
	Lehner's formula identifies the right and left edges of the corresponding free
	semicircular operator as
	\[
		\rho_+=\inf_{Z\succ0}\lambda_{\max}\left(A_0+Z+\sum_{i=1}^n A_iZ^{-1}A_i\right),
	\]
	and
	\[
		\rho_-=\sup_{Z\prec0}\lambda_{\min}\left(A_0+Z+\sum_{i=1}^n A_iZ^{-1}A_i\right).
	\]
	Assuming first that $A_i\succeq0$ for $i\ge1$ and $M\ge N$, we give a
	proof, based on concentration and minimax duality, of the finite-dimensional
	edge bounds
	\[
		\E\lambda_{\max}(H_M)
		\le
		\rho_++O\!\left(\sqrt{\frac{nN}{M}}\,\Arow\right),
		\qquad
		\E\lambda_{\min}(H_M)
		\ge
		\rho_--O\!\left(\sqrt{\frac{nN}{M}}\,\Arow\right).
	\]
	Consequently the same method yields an operator-norm bound with free norm
	constant
	\[
		\rho_*=\max\{\rho_+,-\rho_-\}.
	\]
	In particular, for uniformly bounded positive coefficients, bounded $n$, and
	$N=N_M=o(M)$, this gives
	\[
		\E\|H_M\|_{\op}\le \rho_{*,M}+o(1).
	\]
	Here $\rho_{*,M}=\max\{\rho_{+,M},-\rho_{-,M}\}$ denotes the corresponding free
	norm for the possibly
	$M$-dependent coefficient matrices.  If $\rho_{*,M}\to\rho$, then
	$\limsup_{M\to\infty}\E\|H_M\|_{\op}\le\rho$.
	We then show that the same bounds can be applied to coefficient families
	that are not themselves positive: it is enough that the completely positive
	covariance map
	\[
		\eta(X)=\sum_{i=1}^n A_iXA_i
	\]
	admit a Kraus decomposition by positive semidefinite matrices.  This
	is a checkable certificate for an indefinite displayed family; by covariance
	invariance, it does not enlarge the class of random-matrix laws represented
	by a positive family.
	The proof may be viewed as a matrix-coefficient extension of the classical
	Sudakov--Fernique bounds for Gaussian matrix norms.  It uses a full-process
	Sudakov--Fernique comparison and a
	Davidson--Szarek-type estimate for the largest singular value of an
	$nN\times M$ Gaussian matrix.  The proof uses dual variational formulas for
	Lehner's edge quantities over density matrices.
	We also exhibit the failure, for arbitrary signed Hermitian coefficients, of
	the increment comparison on which the direct Sudakov argument rests.
\end{abstract}

\section{Introduction}

Free probability was introduced by Voiculescu in the 1980s as a noncommutative probability theory in which the notion of independence is replaced by a noncommutative counterpart called freeness.  Although this theory was initially developed in the context of operator algebras, it has found many applications.
One of the first striking connections was with random matrix theory, when
Voiculescu showed that independent GUE matrices become asymptotically free in the limit as their size tends to infinity~\cite{Voiculescu1991}.
Strong asymptotic freeness strengthens this by asserting that the norms of polynomials in independent random matrices converge to the norms of the corresponding free operators in the reduced free product $C^*$-algebra.  This was first proved for GUE matrices by Haagerup and Thorbj\o rnsen~\cite{HaagerupThorbjornsen}, and quantitative forms were later developed by Collins, Guionnet and Parraud~\cite{CollinsParraudGuionnet}.

One common route from polynomial norm convergence to spectral estimates is the
linearization trick, which replaces a polynomial by a larger self-adjoint linear
pencil with matrix coefficients.  However, the strong-convergence proofs
themselves use several different mechanisms.  Haagerup and Thorbj\o rnsen's
original proof uses a master equation for the Stieltjes transform of the
spectrum; Collins, Guionnet and Parraud's approach relies on interpolating
between the GUE and a free product; other techniques use resolvent and
linearization methods~\cite{Anderson2013}, matrix
concentration~\cite{BBvH2023}, and Markov inequalities and interpolation
techniques~\cite{CGVTvHClassicalI,CGVvHClassical,vanHandelSurvey2025}.

The goal of the present paper is different.  We study directly a linear
matrix-coefficient model and show how far one can obtain non-asymptotic
spectral-edge bounds using only Gaussian concentration and comparison
inequalities.  This model is natural in its own right: it is a block Gaussian
matrix with a macroscopic covariance pattern.  The resulting finite-dimensional
error is not optimized, but in the regime of uniformly bounded positive
coefficients and $N=o(M)$ it has the correct free leading constant.  We do not
prove the corresponding asymptotic lower bound.  A subsequent preprint of
Stojnic~\cite{Stojnic2026} claims the analogous asymptotic equality for a real
symmetric Gaussian Kronecker model with positive semidefinite coefficients; the
main setting of that work takes the coefficient dimensions to be fixed.  The
same leading-constant estimate applies to
certain signed families whose covariance map has a positive semidefinite Kraus
decomposition.

Specifically, we focus on the linear matrix-coefficient model
\[
	H_M=A_0\otimes I_M+\frac1{\sqrt M}\sum_{i=1}^n A_i\otimes G_i,
\]
where $A_0,A_1,\ldots,A_n\in M_N(\C)$ are Hermitian.  In the
free limit, $M^{-1/2}G_i$ is replaced by a free semicircular variable
$s_i=l_i+l_i^*$, where $l_i$ are the canonical creation operators on full Fock
space.  The following edge formulas are extracted from Lehner's semicircular
operator-norm formula.

The coefficient family determines the completely positive covariance map
\begin{equation}
	\label{eq:covariance-map}
	\eta(X)=\sum_{i=1}^n A_iXA_i,\qquad X\in M_N(\C).
\end{equation}

\begin{theorem}[Edge formulas from Lehner's semicircular formula]
	\label{thm:lehner}
	Let $l_1,\ldots,l_n$ be the canonical creation operators on full Fock space and
	let $a_0,a_1,\ldots,a_n\in M_m(\C)$ with $a_0=a_0^*$ and $a_i=a_i^*$.  Put
	\[
		x=a_0\otimes I+\sum_{i=1}^n a_i\otimes(l_i+l_i^*).
	\]
	Then
	\[
		\lambda_{\max}(x)
		=
		\inf_{z\succ0}\lambda_{\max}\left(z+a_0+\sum_{i=1}^n a_iz^{-1}a_i\right)
	\]
	and
	\[
		\lambda_{\min}(x)
		=
		\sup_{z\prec0}\lambda_{\min}\left(z+a_0+\sum_{i=1}^n a_iz^{-1}a_i\right).
	\]
\end{theorem}

\begin{proof}
	Here \(\lambda_{\max}(x)=\sup\operatorname{sp}(x)\) and
	\(\lambda_{\min}(x)=\inf\operatorname{sp}(x)\).
	Lehner states an operator-norm formula under the additional assumption
	\(a_0\succeq0\) and \(n\ge2\) in \cite[Corollary~1.5]{Lehner}.  The case
	\(n=1\) follows by adjoining a zero coefficient and an additional free
	semicircular variable.  For arbitrary Hermitian \(a_0\), add a sufficiently
	large scalar multiple \(cI\) so that both \(x+cI\) and the matrices in
	Lehner's variational formula are positive.  The operator norm is then the
	right edge; subtracting \(c\) yields the first formula.  See also
	\cite[Lemma~2.4]{BBvH2023}, whose \(\varepsilon=+1\) term gives the
	corresponding right-edge variational expression after the change of variables
	\(Z\leftrightarrow Z^{-1}\), and Parmaksiz--van
	Handel~\cite{ParmaksizVanHandel2025} for the corresponding edge and
	singular-value formulation.  The left-edge formula follows by applying the
	right-edge formula to \(-x\).
\end{proof}

The quantitative inequalities below use only the variational expressions in
\eqref{eq:rho-plus}--\eqref{eq:rho-minus}; Theorem~\ref{thm:lehner} is needed
solely to identify them with the spectral edges of the free operator.

We apply this with $m=N$ and $a_i=A_i$.  Define the free right and left edges by
\begin{equation}
	\label{eq:rho-plus}
	\rho_+:=\inf_{Z\succ0}\lambda_{\max}\left(A_0+Z+\sum_{i=1}^n A_iZ^{-1}A_i\right)
\end{equation}
and
\begin{equation}
	\label{eq:rho-minus}
	\rho_-:=\sup_{Z\prec0}\lambda_{\min}\left(A_0+Z+\sum_{i=1}^n A_iZ^{-1}A_i\right).
\end{equation}
We also write
\[
	\rho_*=\max\{\rho_+,-\rho_-\}
\]
for the norm of the corresponding free self-adjoint operator, by
Theorem~\ref{thm:lehner}.
Let
\[
	\DN=\{C\in M_N(\C):C\succeq0,\ \Tr C=1\}
\]
be the density matrices, and, for \(C\in\DN\), set
\begin{equation}
	\label{eq:EC}
	E_C
	=
	C^{1/2}\eta(C)C^{1/2}
	=
	\sum_{i=1}^n(C^{1/2}A_iC^{1/2})^2.
\end{equation}
Our goal is to show
how far one can go using only Gaussian concentration and comparison inequalities.
The main result is the following finite-dimensional expected operator-norm
estimate.

\begin{restatable*}[Main theorem]{theorem}{mainthm}
	\label{thm:main}
	Assume that
	\[
		A_i\succeq0,\qquad i=1,\ldots,n,
	\]
	and assume $M\ge N$.  Then
	\[
		\E\|H_M\|_{\op}
		\le
		\rho_*+9\sqrt{\frac{nN}{M}}\,\Arow.
	\]
\end{restatable*}

The constant \(9\) in Theorem~\ref{thm:main} is universal: it does not depend on
\(n,N,M\), or on the coefficient matrices, and no attempt is made to optimize
it.

The covariance-map formulation gives the following extension, which is proved
in Section~\ref{sec:rekraus}.

\begin{restatable*}[Bound from a positive semidefinite Kraus decomposition]
	{theorem}{psdrekrausthm}
	\label{thm:psd-rekraus}
	Assume $M\ge N$, and suppose that the covariance map in
	\eqref{eq:covariance-map} admits a representation
	\[
		\eta(X)=\sum_{j=1}^r\widetilde A_jX\widetilde A_j,
		\qquad \widetilde A_1,\ldots,\widetilde A_r\succeq0.
	\]
	Then
	\[
		\E\lambda_{\max}(H_M)
		\le
		\rho_+
		+6\sqrt{\frac{rN}{M}}\,\|\eta(I_N)\|_{\op}^{1/2},
	\]
	\[
		\E\lambda_{\min}(H_M)
		\ge
		\rho_-
		-6\sqrt{\frac{rN}{M}}\,\|\eta(I_N)\|_{\op}^{1/2},
	\]
	and
	\[
		\E\|H_M\|_{\op}
		\le
		\rho_*
		+9\sqrt{\frac{rN}{M}}\,\|\eta(I_N)\|_{\op}^{1/2}.
	\]
	Here \(r\) is the length of the chosen positive semidefinite Kraus
	decomposition.  The leading terms are the Lehner edges of the original
	displayed family; only the finite-dimensional error records \(r\).
\end{restatable*}

This criterion is nonvacuous even for very small indefinite families.  For
example, with
\[
	\sigma_x=\begin{pmatrix}0&1\\1&0\end{pmatrix},
\]
the family \((\sigma_x,I_2)\) is transformed by a real orthogonal change of
generators into
\[
	\left(\frac{I_2+\sigma_x}{\sqrt2},
	\frac{I_2-\sigma_x}{\sqrt2}\right),
\]
whose two members are positive semidefinite.  Thus
Theorem~\ref{thm:psd-rekraus} applies with \(r=2\), even though the displayed
coefficient \(\sigma_x\) is indefinite.

\begin{corollary}[Asymptotic consequence]
	\label{cor:asymptotic}
	Let \(n\) be fixed.  For each \(M\), let
	\(A_{0,M},A_{1,M},\ldots,A_{n,M}\in M_{N_M}(\C)\) be Hermitian matrices with
	\(A_{i,M}\succeq0\) for \(i=1,\ldots,n\), and assume \(N_M=o(M)\).  Suppose
	there is a constant \(K\) such that
	\[
		\max_{1\le i\le n}\|A_{i,M}\|_{\op}\le K
	\]
	for all \(M\).  Let \(\rho_{+,M}\), \(\rho_{-,M}\) be the free right and left
	edges defined in \eqref{eq:rho-plus}--\eqref{eq:rho-minus} for the coefficient
	family \(A_{0,M},A_{1,M},\ldots,A_{n,M}\), and set
	\[
		\rho_{*,M}=\max\{\rho_{+,M},-\rho_{-,M}\}.
	\]
	Let \(H_M\) be formed from these coefficients as in Theorem~\ref{thm:main}.
	Then, for all sufficiently large \(M\) (for which \(N_M=o(M)\) ensures
	\(M\ge N_M\)),
	\[
		\E\|H_M\|_{\op}\le \rho_{*,M}+o(1).
	\]
	In particular, if \(\rho_{*,M}\to\rho\), then
	\[
		\limsup_{M\to\infty}\E\|H_M\|_{\op}\le\rho .
	\]
\end{corollary}

\begin{proof}
	By Theorem~\ref{thm:main},
	\[
		\E\|H_M\|_{\op}
		\le
		\rho_{*,M}
		+
		9\sqrt{\frac{nN_M}{M}}
		\left\|\sum_{i=1}^n A_{i,M}^2\right\|_{\op}^{1/2}.
	\]
	The last factor is at most \(\sqrt n\,K\), while \(n\) is fixed and
	\(N_M=o(M)\).  Hence the error term is \(o(1)\).
\end{proof}

When $n=N=1$ and $A_0=0$, the upper-edge part of our argument reduces to the
classical upper bound
\[
	\E\lambda_{\max}(M^{-1/2}G)
	\le 2+O(M^{-1/2})
\]
for a normalized GUE matrix $G$.  This scalar case is closely related to the
classical concentration-and-comparison estimates for Gaussian matrix norms; see,
for example, Davidson--Szarek~\cite[Section~II.c]{DavidsonSzarek2001} and the
2018 first edition of
Vershynin~\cite[Section~7.3]{Vershynin}.
For inhomogeneous matrices with independent entries, see also
van Handel~\cite[Section~4]{vanHandel2017SpectralNorm}.  In the present proof,
the Davidson--Szarek estimate enters at the end through a rectangular
block-Gaussian singular-value bound, while the full-process comparison is a
separate step.  Sharper non-asymptotic edge deviation estimates for the GUE,
which in particular imply
$\E\lambda_{\max}(M^{-1/2}G)\to2$, were obtained by
Ledoux~\cite{LedouxLargestEigenvalues} and Aubrun~\cite{AubrunLargestEigenvalue}.
The factor \(\Arow\) in the error term is the same row/column square-function
quantity that appears in the endpoint noncommutative Khintchine inequality.  For
a free semicircular family \((s_i)_{i=1}^n\), the classical operator-space form
of this inequality gives, for a universal constant \(K\),
\[
	\left\|\sum_i a_i\otimes s_i\right\|
	\le
	K\max\left\{
	\left\|\sum_i a_i^*a_i\right\|^{1/2},
	\left\|\sum_i a_ia_i^*\right\|^{1/2}
	\right\},
\]
see, for example, Pisier~\cite[(9.9.8)--(9.9.9)]{PisierOperatorSpace} and
Buchholz~\cite[Remark~1]{BuchholzKhintchine}; for self-adjoint coefficients
this reduces precisely to \(\Arow\).
Another point of the argument that may be of independent interest is
Propositions~\ref{prop:dual} and~\ref{prop:leftdual}, which give dual forms of
Lehner's edge quantities over density matrices.  After completion of this work,
we learned of independent related work of Kunisky~\cite{KuniskyLehnerSDP2026},
which develops semidefinite-programming formulations of Lehner's formulas and
also discusses such dual formulations.

The proof of Theorem~\ref{thm:main} proceeds through separate estimates for the
two spectral edges.  The first ingredient is the following upper-edge estimate.

\begin{restatable*}[Upper-edge estimate]{proposition}{upperedgeprop}
	\label{prop:upperedge}
	Assume that
	\[
		A_i\succeq0,\qquad i=1,\ldots,n,
	\]
	and assume $M\ge N$.  Then
	\[
		\E\lambda_{\max}(H_M)
		\le
		\rho_++6\sqrt{\frac{nN}{M}}\,\Arow.
	\]
\end{restatable*}

The second ingredient is the corresponding lower-edge estimate.  The dual
variational formula, proved in Proposition~\ref{prop:leftdual}, reads
\[
	\rho_-=
	\min_{C\in\DN}
	\left\{
	\Tr(A_0C)-2\Tr(E_C^{1/2})
	\right\}.
\]

\begin{restatable*}[Left-edge estimate]{proposition}{leftedgeprop}
	\label{prop:leftedge}
	Assume that
	\[
		A_i\succeq0,\qquad i=1,\ldots,n,
	\]
	and assume $M\ge N$.  Then
	\[
		\E\lambda_{\min}(H_M)
		\ge
		\rho_-
		-
		6\sqrt{\frac{nN}{M}}\,\Arow.
	\]
\end{restatable*}

The asymptotic consequence of Theorem~\ref{thm:main} should be viewed in the
context of the strong convergence phenomenon for random matrices; see
\cite{vanHandelSurvey2025} for a recent survey.  In its greatest generality, this
asymptotic statement is not new.  Recent work of Chen--Garza-Vargas--van Handel
\cite[Theorem~1.1 and Corollary~1.2]{CGVvHClassical} proves strong convergence
for polynomials in GUE, GOE, and GSE matrices with matrix coefficients of
dimension $D_M=\exp(o(M))$.  This covers, in particular, the regime $N=o(M)$
considered here, and does so even for signed coefficients and arbitrary
fixed-degree noncommutative polynomials.  This work builds on the general
framework introduced in~\cite{CGVTvHClassicalI}.  In the earlier history of
growing coefficient dimensions, Pisier observed that the quantitative
Haagerup--Thorbj\o rnsen method yields
\(N=o(M^{1/4})\)~\cite[(0.9)]{PisierSubexponentialOperatorSpaces}, while the
matrix-concentration method of Bandeira--Boedihardjo--van Handel reaches
\(N=o(M/\log^3M)\) for arbitrary signed
coefficients~\cite[Theorem~2.10]{BBvH2023}.  Other results include
Collins--Guionnet--Parraud~\cite{CollinsParraudGuionnet}, which gave the
tensor-coefficient regime $N=o(M^{1/3})$, and the subsequent work of
Parraud~\cite{ParraudTensor2024}, which reaches $N=\exp(o(M^{2/3}))$.

There are two relevant non-asymptotic comparisons.  First, for coefficient
dimension \(N=M^{O(1)}\), \cite[Corollary~1.4]{CGVvHClassical} gives the
one-sided relative upper estimate
\[
	\|P_M(G^M)\|_{\op}
	\le
	\left(
		1+O_{\mathrm P}\!\left(\sqrt{\frac{\log M}{M}}\right)
	\right)
	\|P_M(s)\|_{\op}
\]
for arbitrary signed coefficients and, more generally, for fixed-degree
noncommutative polynomials.  The displayed rate contains no explicit
coefficient-dimension factor once the polynomial growth regime is fixed.  In
additive terms, its error scale is
\[
	O_{\mathrm P}\!\left(
	\sqrt{\frac{\log M}{M}}\,\|P_M(s)\|_{\op}
	\right).
\]
Consequently, when \(\|P_M(s)\|_{\op}\asymp\sigma\), comparison with our bound
suggests the crossover \(rN\asymp\log M\).  No such unconditional crossover
holds when, for example, the deterministic term \(A_0\) dominates.

Second, \cite[Corollary~2.3]{BCSvH2026}, applied to the present model, gives,
with \(\sigma=\|\sum_iA_i^2\|_{\op}^{1/2}\), a two-sided estimate of order
\[
	C\sigma\left(\frac NM\right)^{1/4}
	\bigl(\log(NM)\bigr)^{3/4}
\]
for a universal constant \(C\), both for the norm and for either spectral edge.
It applies to arbitrary signed coefficients and does not require \(M\ge N\).
In comparison, for a linear self-adjoint pencil whose covariance map admits a
positive semidefinite Kraus decomposition of length \(r\), our direct
concentration argument gives the one-sided finite-dimensional estimate
\[
	\E\|H_M\|_{\op}
	\le
	\left\|A_0\otimes I+\sum_i A_i\otimes s_i\right\|_{\op}
	+
	O\!\left(
	\sqrt{\frac{rN}{M}}\left\|\eta(I_N)\right\|_{\op}^{1/2}
	\right),
\]
where the free norm is identified explicitly by Lehner's variational formula.
Thus our error has the better exponent \(M^{-1/2}\), no logarithmic factor, and
an explicit constant, at the price of a one-sided conclusion and the
positive-semidefinite-Kraus restriction.  Compared with the one-sided relative
upper estimate of \cite[Corollary~1.4]{CGVvHClassical}, it also has three useful
features: it is additive rather than multiplicative; it is an expectation bound
with an explicit constant rather than a bound in probability with an unspecified
one; and its error term involves only the covariance part, not \(A_0\), so it
remains effective in the regime \(\|A_0\|_{\op}\gg\sigma\) allowed by
Corollary~\ref{cor:asymptotic}.  It gives a short quantitative route to the
$N=o(M)$ upper bound, without a master equation or resolvent asymptotic
expansions.  The displayed coefficients themselves need not be positive, but
the covariance invariance in Section~\ref{sec:rekraus} shows that this criterion
is a certificate for applying the positive-family theorem, not a new class of
random-matrix laws.  The result should therefore not be read as an input to a
linearization proof of general strong convergence, where signed coefficients
and full spectral convergence are unavoidable.  Rather, the geometry of the
underlying Gaussian processes already produces the free-probability edge
formulas for this class of block models.

This paper is organized as follows.  The next section sets up notation and states
some preliminary lemmas.  Section~\ref{sec:variational} proves dual variational
formulas for $\rho_+$ and $\rho_-$ over density matrices.
Section~\ref{sec:full-process} proves the full-process comparison estimate of
Proposition~\ref{prop:fullprocess}, which is the main technical ingredient in the
proof of Proposition~\ref{prop:upperedge}.  Section~\ref{sec:block-wishart} gives
the block-Wishart estimate needed to control the comparison process and proves
Proposition~\ref{prop:upperedge}.  Section~\ref{sec:left-edge} records the
corresponding lower-edge estimate.  Section~\ref{sec:op-norm} establishes the
Lipschitz control of the operator norm and proves Theorem~\ref{thm:main}.
Section~\ref{sec:rekraus} uses the covariance invariance of the model to extend
the theorem to maps admitting a positive semidefinite Kraus decomposition.
Finally, Section~\ref{sec:signed} identifies the failure of the specific
increment comparison for arbitrary signed Hermitian coefficients.

\section{Setup and preliminaries}
\label{sec:setup}

Throughout, $A_0,A_1,\ldots,A_n\in M_N(\C)$ are Hermitian.  Positivity
assumptions are stated explicitly in the results where they are used.  We write
$\|\cdot\|_{\op}$ for the operator norm,
$\|X\|_1=\Tr|X|$ for the trace norm, and
\[
	\ip{X}{Y}_F=\Rea\Tr(X^*Y)
\]
for the real Frobenius inner product.

The GUE matrices $G_i$ are normalized by
\[
	\E\ip{G_i}{K}_F^2=\|K\|_F^2
\]
for every Hermitian $K\in M_M(\C)$.  Equivalently, diagonal entries are real
standard Gaussians and off-diagonal real and imaginary parts have variance
$1/2$.  Throughout the paper, a standard complex Gaussian means a random
variable of the form
\[
	X+iY,
\]
where $X,Y$ are independent real Gaussians with law $N(0,1/2)$, and
$N(\mu,\sigma^2)$ denotes the real Gaussian distribution with mean $\mu$ and
variance $\sigma^2$.  Thus, if $\Gamma$ is a standard complex Gaussian matrix
and $L$ is a deterministic matrix of the same size, then
\[
	\E\ip{\Gamma}{L}_F^2
	=
	\frac12\|L\|_F^2.
\]
Equivalently, $\E|\Tr(\Gamma^*L)|^2=\|L\|_F^2$.

Recall that \(\DN\) denotes the density matrices.  With \(E_C\) as in
\eqref{eq:EC}, define
\[
	\phi(C)=\Tr(A_0C)+2\Tr(E_C^{1/2}).
\]
We identify $\C^N\otimes\C^M$ with $\C^{N\times M}$ by
\[
	\vecop(S)=\sum_{a,b}S_{ab}(e_a\otimes e_b).
\]
Then
\[
	(A\otimes G)\vecop(S)=\vecop(ASG^{\top}).
\]
Consequently, for Hermitian $A$ and $G$,
\[
	\ip{\vecop(S)}{(A\otimes G)\vecop(S)}
	=\Tr(S^*ASG^{\top}).
\]
Here the unsubscripted brackets denote the standard complex Hilbert-space
inner product on \(\C^N\otimes\C^M\).

For $M\ge N$ set
\[
	\St_{N,M}=\{W\in M_{N,M}(\C): WW^*=I_N\}.
\]
Thus elements of $\St_{N,M}$ are coisometries.

\begin{lemma}[Variational characterization of the trace norm]
	\label{lem:trace-norm-variational}
	Let $A\in M_{N,M}(\C)$ and assume $M\ge N$.  Then
	\[
		\|A\|_1=\sup_{W\in\St_{N,M}}\Rea\Tr(A^*W).
	\]
\end{lemma}
\begin{proof}
	This is an immediate consequence of the duality between the Schatten
	$1$-norm and the spectral norm~\cite[(1.173)]{Watrous2018}, together with the
	singular value theorem~\cite[Theorem~1.6]{Watrous2018}.  Equivalently, it
	follows from von Neumann's trace inequality in the form of
	\cite[Theorem~IV.2.5]{Bhatia1997}.  The supremum is attained, for instance
	by taking the partial isometry from the singular value decomposition of
	\(A\) and extending it to an element of \(\St_{N,M}\).
\end{proof}

\begin{definition}[Gaussian process]
	Let $T$ be an index set. A (real-valued) Gaussian process indexed by $T$ is a
	family of real random variables $(X_t)_{t\in T}$ such that for every
	$t_1,\ldots,t_k\in T$, the random vector
	\[
		(X_{t_1},\ldots,X_{t_k})
	\]
	is jointly Gaussian. The process is centered if $\E X_t=0$ for every
	$t\in T$.
\end{definition}

\begin{theorem}[Sudakov--Fernique with deterministic drift]
	\label{thm:sudakov-fernique-drift}
	Let $T$ be a compact metric space. Let $(X_t)_{t\in T}$ and
	$(Y_t)_{t\in T}$ be centered Gaussian processes whose sample paths are
	almost surely continuous. Suppose that for every $s,t\in T$,
	\[
		\E\bigl[(X_s-X_t)^2\bigr]
		\le
		\E\bigl[(Y_s-Y_t)^2\bigr].
	\]
	Then, for every continuous deterministic function $a:T\to\mathbb R$,
	\[
		\E \sup_{t\in T}\{a(t)+X_t\}
		\le
		\E \sup_{t\in T}\{a(t)+Y_t\}.
	\]
\end{theorem}
\begin{proof}
	This is Vitale's deterministic-shift extension of the
	Sudakov--Fernique comparison: see \cite[(4), Section 1]{Vitale}.
	Strictly speaking, Vitale states the result for denumerable index sets;
	the compact-continuous version follows by restricting to a countable
	dense subset of \(T\), using the continuity of the sample paths.
\end{proof}

Our use of Theorem~\ref{thm:sudakov-fernique-drift} follows a fixed strategy.
First, we rewrite $\lambda_{\max}(H_M)$ as the supremum of a Gaussian process
indexed by density matrices and Stiefel coisometries; see
Lemma~\ref{lem:density-stiefel}.  We then compare the original process
$X_{C,W}$ in \eqref{eq:original-process} with the comparison process
$Y_{C,W}$ in \eqref{eq:comparison-process}, whose supremum can be evaluated
through the trace-norm variational formula.
Second, we verify the increment domination required in
Theorem~\ref{thm:sudakov-fernique-drift}.  For positive coefficients this
verification is reduced to the deterministic two-fiber estimate of
Lemma~\ref{lem:twofiber}, and is carried out in the proof of
Proposition~\ref{prop:fullprocess}.

\section{Variational formulas over density matrices}
\label{sec:variational}

The goal of this section is to prove dual variational formulas for the free
edges $\rho_+$ and $\rho_-$ over density matrices.  These formulas identify the
deterministic functional produced by the comparison process with the
variational edge quantities.

In this section, we will need to exchange the order of a supremum and an infimum in a variational formula. For this, we recall the following minimax theorem of Sion~\cite{Sion}.
The minimax step below may be viewed as a matrix-coefficient analogue of the
scalar/diagonal minimax calculation in \cite[Lemma~3.2]{BBvH2023}.

\begin{lemma}[Sion's minimax theorem]
	\label{lem:sion}
	Let $X$ be a convex subset of a topological vector space,
	$Y$ be a convex subset of a topological vector space,
	and suppose either $X$ or $Y$ is compact.
	Suppose
	$f:X\times Y\to\mathbb R$ is such that, for each $y\in Y$, the map
	$x\mapsto f(x,y)$ is upper semicontinuous and concave, and for each
	$x\in X$, the map $y\mapsto f(x,y)$ is lower semicontinuous and convex.
	Then
	\[
		\sup_{x\in X}\inf_{y\in Y} f(x,y)
		=
		\inf_{y\in Y}\sup_{x\in X} f(x,y).
	\]
\end{lemma}
\begin{proof}
	See Sion~\cite{Sion}.
\end{proof}

For $Z\succ0$ define
\[
	B(Z)=A_0+Z+\sum_{i=1}^n A_iZ^{-1}A_i.
\]
Writing
\[
	f(C,Z)=\Tr(A_0C)+\Tr(CZ)+\Tr(\eta(C)Z^{-1}),
\]
we have
\[
	\rho_+=\inf_{Z\succ0}\sup_{C\in\DN}f(C,Z).
\]

\begin{lemma}[Weighted trace AM--GM]
	\label{lem:weighted-amgm}
	For $C,D\succeq0$,
	\[
		\inf_{Z\succ0}\{\Tr(CZ)+\Tr(DZ^{-1})\}
		=2\Tr\bigl[(C^{1/2}DC^{1/2})^{1/2}\bigr].
	\]
\end{lemma}

\begin{proof}
	If $C\succ0$, put $Y=C^{1/2}ZC^{1/2}$.  Then
	\[
		\Tr(CZ)+\Tr(DZ^{-1})
		=\Tr Y+\Tr(C^{1/2}DC^{1/2}Y^{-1}).
	\]
	The matrix trace AM--GM identity
	\[
		\inf_{Y\succ0}\{\Tr Y+\Tr(EY^{-1})\}=2\Tr(E^{1/2})
	\]
	follows by expanding
	\[
		\Tr\bigl[(Y^{1/2}-E^{1/2}Y^{-1/2})(Y^{1/2}-E^{1/2}Y^{-1/2})^*\bigr]\ge0
	\]
	and then taking $Y=E^{1/2}+\eta I$, $\eta\downarrow0$.

	For singular $C\succeq0$, the upper bound follows from the positive-definite
	case applied to $C+\eta I$.  Since $C+\eta I\succeq C$,
	\[
		\inf_{Z\succ0}\{\Tr(CZ)+\Tr(DZ^{-1})\}
		\le
		\inf_{Z\succ0}\{\Tr((C+\eta I)Z)+\Tr(DZ^{-1})\}
		=
		2\Tr\bigl[((C+\eta I)^{1/2}D(C+\eta I)^{1/2})^{1/2}\bigr],
	\]
	and then $\eta\downarrow0$.  For the lower bound, apply the positive-definite
	case to $C+\eta I$.  For every fixed $Z\succ0$,
	\[
		\Tr(CZ)+\Tr(DZ^{-1})+\eta\Tr(Z)
		\ge
		2\Tr\bigl[((C+\eta I)^{1/2}D(C+\eta I)^{1/2})^{1/2}\bigr].
	\]
	Letting $\eta\downarrow0$ gives the desired lower bound for this fixed $Z$,
	and then taking the infimum over $Z$ proves the singular case.
\end{proof}

\begin{proposition}[Dual form of Lehner's right edge]
	\label{prop:dual}
	For Hermitian $A_0,A_1,\ldots,A_n$,
	\[
		\rho_+=\,\max_{C\in\DN}\phi(C).
	\]
\end{proposition}

\begin{proof}
	For $Z\succ0$,
	\[
		\lambda_{\max}(B(Z))=\sup_{C\in\DN}\Tr(CB(Z))=\sup_{C\in\DN}f(C,Z).
	\]
	The set $\DN$ is compact and convex, and the positive-definite cone is convex.
	Moreover, $f$ is affine in $C$ and convex in $Z$.  Lemma~\ref{lem:sion},
	applied with the compact variable $C$, gives
	\[
		\rho_+=
		\inf_{Z\succ0}\sup_{C\in\DN}f(C,Z)
		=
		\sup_{C\in\DN}\inf_{Z\succ0}f(C,Z).
	\]
	For fixed $C$, Lemma~\ref{lem:weighted-amgm} gives
	\[
		\inf_{Z\succ0}f(C,Z)
		=\Tr(A_0C)+2\Tr\bigl[(C^{1/2}\eta(C)C^{1/2})^{1/2}\bigr]
		=\phi(C).
	\]
	Therefore
	\[
		\rho_+=\sup_{C\in\DN}\phi(C).
	\]
	The supremum is attained because $\DN$ is compact and $\phi$ is continuous.
\end{proof}

\begin{example}[A two-dimensional coordinate pencil]
	\label{ex:dual-coordinate}
	Let \(N=n=2\), \(A_0=0\), and
	\[
		A_1=\begin{pmatrix}1&0\\0&0\end{pmatrix},
		\qquad
		A_2=\begin{pmatrix}0&0\\0&1\end{pmatrix}.
	\]
	For
	\[
		C=\begin{pmatrix}p&z\\\overline z&1-p\end{pmatrix}\in\mathcal D_2,
	\]
	we have \(\eta(C)=\operatorname{diag}(p,1-p)\).  For a positive
	\(2\times2\) matrix \(E\),
	\((\Tr E^{1/2})^2=\Tr E+2\sqrt{\det E}\); hence
	\[
		\begin{aligned}
			\bigl(\Tr E_C^{1/2}\bigr)^2
			 & =p^2+(1-p)^2
			+2\sqrt{p(1-p)\bigl(p(1-p)-|z|^2\bigr)} \\
			 & \le 1.
		\end{aligned}
	\]
	Equality holds whenever \(z=0\), in particular at \(C=I_2/2\).
	Proposition~\ref{prop:dual} therefore gives
	\(\rho_+=\max_C\phi(C)=2\), in agreement with the fact that the free
	pencil is \(\operatorname{diag}(s_1,s_2)\).
\end{example}

\begin{proposition}[Dual form of Lehner's left edge]
	\label{prop:leftdual}
	For Hermitian $A_0,A_1,\ldots,A_n$,
	\[
		\rho_-=
		\min_{C\in\DN}
		\left\{
		\Tr(A_0C)-2\Tr(E_C^{1/2})
		\right\}.
	\]
\end{proposition}

\begin{proof}
	Directly from the definitions \eqref{eq:rho-plus}--\eqref{eq:rho-minus},
	the substitution \(Z=-Y\), together with
	\(\lambda_{\min}(-B)=-\lambda_{\max}(B)\), gives
	\[
		\rho_-(A_0,A_1,\ldots,A_n)=-\rho_+(-A_0,A_1,\ldots,A_n).
	\]
	Applying Proposition~\ref{prop:dual} with deterministic coefficient $-A_0$
	gives
	\[
		\rho_-=
		-\max_{C\in\DN}
		\left\{
		-\Tr(A_0C)+2\Tr(E_C^{1/2})
		\right\},
	\]
	which is the stated minimum.
\end{proof}

\begin{remark}
	The use of Sion's minimax theorem in Proposition~\ref{prop:dual} is not needed
	for the main theorem.  The estimates below only require the upper bound
	\[
		\max_{C\in\DN}\phi(C)
		=
		\sup_{C\in\DN}\inf_{Z\succ0}f(C,Z)
		\le
		\inf_{Z\succ0}\sup_{C\in\DN}f(C,Z)
		=
		\rho_+,
	\]
	which is just the elementary direction of the minimax inequality.  Sion's
	theorem is used only to identify this upper bound with the full dual formula.
\end{remark}

\section{Full-process comparison}
\label{sec:full-process}

The goal of this section is to prove Proposition~\ref{prop:fullprocess}.  Namely,
we introduce the two processes that will be compared by
Theorem~\ref{thm:sudakov-fernique-drift} and verify the required increment
domination.

For $C\in\DN$ put
\[
	B_i(C)=C^{1/2}A_iC^{1/2},
	\qquad
	E_C=\sum_{i=1}^n B_i(C)^2.
\]
This agrees with the notation in \eqref{eq:EC}.

\begin{lemma}[Two-fiber Stiefel increment]
	\label{lem:twofiber}
	Let $P,Q\succeq0$ be $N\times N$ matrices and let $W,V\in\St_{N,M}$.
	Then
	\[
		\|W^*PW-V^*QV\|_F^2\le 2\|PW-QV\|_F^2.
	\]
\end{lemma}

\begin{proof}
	This is a direct application of Araki--Yamagami's Hilbert--Schmidt
	inequality~\cite[Theorem~1]{ArakiYamagami1981},
	\[
		\||X|-|Y|\|_F\le \sqrt2\,\|X-Y\|_F.
	\]
	The cited square-matrix statement applies to rectangular matrices by
	embedding \(X\in M_{N,M}(\C)\) as the upper-right block of an
	\((N+M)\times(N+M)\) matrix with all other blocks zero.  Under this
	embedding, both \(\|X-Y\|_F\) and \(\||X|-|Y|\|_F\) are unchanged.
	Now, since $WW^*=VV^*=I_N$ and $P,Q\succeq0$,
	\[
		|PW|=(W^*P^2W)^{1/2}=W^*PW,\qquad
		|QV|=(V^*Q^2V)^{1/2}=V^*QV.
	\]
	Applying the cited inequality with $X=PW$ and $Y=QV$ gives the claim.
\end{proof}

\begin{lemma}[Density--Stiefel representation]
	\label{lem:density-stiefel}
	Assume \(M\ge N\).  Then, in distribution,
	\begin{equation}
		\label{eq:lambdamax-density-stiefel}
		\lambda_{\max}(H_M)
		=
		\sup_{C\in\DN,\,W\in\St_{N,M}}
		\left\{
		\Tr(A_0C)+\frac1{\sqrt M}\sum_{i=1}^n
		\ip{G_i}{W^*B_i(C)W}_F
		\right\}.
	\end{equation}
\end{lemma}

\begin{proof}
	The Rayleigh-quotient formula and the vectorization convention of
	Section~\ref{sec:setup} give
	\[
		\lambda_{\max}(H_M)
		=
		\sup_{\|S\|_F=1}
		\left\{
		\Tr(S^*A_0S)+\frac1{\sqrt M}\sum_i
		\Tr(S^*A_iS G_i^{\top})
		\right\}.
	\]
	For such \(S\), put \(C=SS^*\in\DN\).  The polar decomposition, with
	the partial coisometry extended arbitrarily on \(\ker C\), gives
	\(S=C^{1/2}W\) for some \(W\in\St_{N,M}\).  Conversely, every pair
	\((C,W)\in\DN\times\St_{N,M}\) yields \(S=C^{1/2}W\) with
	\(\|S\|_F^2=\Tr C=1\), so this parametrization is surjective.  Finally,
	\[
		\Tr(S^*A_0S)=\Tr(A_0C),\qquad
		S^*A_iS=W^*B_i(C)W.
	\]
	The assertion follows because
	\((G_1,\ldots,G_n)\overset{d}{=}
	(G_1^{\top},\ldots,G_n^{\top})\).
\end{proof}

Let $\Gamma_1,\ldots,\Gamma_n\in M_{N,M}(\C)$ be independent standard complex
Gaussian matrices.  We compare the centered Gaussian processes indexed by
$(C,W)\in\DN\times\St_{N,M}$ defined by
\begin{equation}
	\label{eq:original-process}
	X_{C,W}=\frac1{\sqrt M}\sum_{i=1}^n
	\ip{G_i}{W^*B_i(C)W}_F
\end{equation}
and
\begin{equation}
	\label{eq:comparison-process}
	Y_{C,W}=\frac2{\sqrt M}\sum_{i=1}^n\ip{\Gamma_i}{B_i(C)W}_F.
\end{equation}

\begin{proposition}[Full-process Sudakov comparison]
	\label{prop:fullprocess}
	Assume $M\ge N$ and $A_i\succeq0$ for $i=1,\ldots,n$.
	Then
	\[
		\E\lambda_{\max}(H_M)
		\le
		\E\sup_{C\in\DN}
		\left\{
		\Tr(A_0C)+\frac2{\sqrt M}\left\|\sum_{i=1}^nB_i(C)\Gamma_i\right\|_1
		\right\}.
	\]
\end{proposition}

\begin{proof}
	The index set \(\DN\times\St_{N,M}\) is compact: both factors are closed
	and bounded subsets of finite-dimensional matrix spaces.  The map
	\(C\mapsto C^{1/2}\) is continuous, so, for every realization of the
	Gaussian matrices, the sample paths of \(X_{C,W}\) and \(Y_{C,W}\) are
	continuous in \((C,W)\).  The deterministic drift
	\((C,W)\mapsto\Tr(A_0C)\) is continuous as well.  Thus all the topological
	hypotheses of Theorem~\ref{thm:sudakov-fernique-drift} are satisfied.

	Since $A_i\succeq0$, also $B_i(C)\succeq0$ for every $C\in\DN$.
	For two indices $(C,W)$ and $(C',V)$,
	\[
		\E|X_{C,W}-X_{C',V}|^2
		=\frac1M\sum_{i=1}^n
		\|W^*B_i(C)W-V^*B_i(C')V\|_F^2.
	\]
	By Lemma~\ref{lem:twofiber}, this is at most
	\[
		\frac2M\sum_{i=1}^n\|B_i(C)W-B_i(C')V\|_F^2.
	\]
	On the other hand, the normalization of the complex Gaussian entries gives
	\[
		\E|Y_{C,W}-Y_{C',V}|^2
		=\frac2M\sum_{i=1}^n\|B_i(C)W-B_i(C')V\|_F^2.
	\]
	Applied with the deterministic drift $(C,W)\mapsto\Tr(A_0C)$, the same theorem
	gives the inequality used here:
	\[
		\E\sup_{C,W}\{\Tr(A_0C)+X_{C,W}\}
		\le
		\E\sup_{C,W}\{\Tr(A_0C)+Y_{C,W}\}.
	\]
	For fixed $C$ and fixed $\Gamma_i$,
	\[
		\sup_{W\in\St_{N,M}}\Rea\sum_i\Tr(\Gamma_i^*B_i(C)W)
		=\left\|\sum_iB_i(C)\Gamma_i\right\|_1,
	\]
	by Lemma~\ref{lem:trace-norm-variational}.  This proves the claim.
\end{proof}

\section{The block-Wishart estimate and the upper edge}
\label{sec:block-wishart}

The full-process comparison still contains a supremum over $C$ inside the
expectation.  The following proposition controls it uniformly by using a single
block Gaussian matrix.  We state it first with the quantity $\Psi$ defined below, rather than
with $\rho_+$, because this form will also be used for the lower edge with
deterministic coefficient $-A_0$.

\begin{lemma}[Positive-part estimate for the largest Wishart eigenvalue]
	\label{lem:wishart-estimate}
	Let $\Gamma$ be a $d\times M$ random matrix with independent standard real or
	standard complex Gaussian entries.  Then
	\[
		\E\left[
			\left(
			\left\|\frac{\Gamma\Gamma^*}{M}\right\|_{\op}^{1/2}
			-1
			\right)_+
			\right]
		\le
		3\sqrt{\frac{d}{M}} .
	\]
\end{lemma}

\begin{proof}
	We first record the tail bound
	\[
		\mathbb P\left\{
		\|\Gamma\|_{\op}\ge \sqrt M+\sqrt d+t
		\right\}
		\le e^{-t^2/2}.
	\]
	For real Gaussian entries this is the non-asymptotic singular-value estimate
	of Davidson and Szarek~\cite[Theorem~II.13]{DavidsonSzarek2001}, after
	rescaling their normalization.  If needed, transpose the matrix to match the
	aspect-ratio convention in their statement; the estimate and the quantity
	\(\sqrt M+\sqrt d\) are symmetric under transposition.  For standard complex
	Gaussian entries in the
	convention \(\E|\gamma|^2=1\), the
	expectation bound
	\[
		\E\|\Gamma\|_{\op}\le \sqrt M+\sqrt d
	\]
	follows from the complex Hilbert-space form of Chevet's
	inequality~\cite[Section~43]{TomczakJaegermann1989}.  Since the map
	$\Gamma\mapsto\|\Gamma\|_{\op}$ is 1-Lipschitz for the real Frobenius norm,
	Gaussian concentration on the underlying real coordinates gives the displayed
	tail bound; with our complex normalization the concentration exponent is in
	fact $e^{-t^2}$, so the weaker bound above follows.
	Since
	\[
		\left\|\frac{\Gamma\Gamma^*}{M}\right\|_{\op}^{1/2}
		=
		\frac{\|\Gamma\|_{\op}}{\sqrt M},
	\]
	the tail-integral formula gives
	\[
		\begin{aligned}
			\E\left[
				\left(
				\left\|\frac{\Gamma\Gamma^*}{M}\right\|_{\op}^{1/2}
				-1
				\right)_+
				\right]
			 & \le
			\sqrt{\frac dM}
			+
			\frac1{\sqrt M}\int_0^\infty e^{-u^2/2}\,du \\
			 & =
			\sqrt{\frac dM}
			+
			\sqrt{\frac{\pi}{2M}}
			\le
			3\sqrt{\frac dM},
		\end{aligned}
	\]
	where the last inequality uses $d\ge1$.
\end{proof}

\begin{proposition}[Uniform control of the comparison process]
	\label{prop:blockwishart}
	Assume $M\ge N$ and $A_i\succeq0$ for $i=1,\ldots,n$.
	Define
	\[
		\begin{aligned}
			\Psi
			         & :=
			\max_{C\in\DN}
			\left\{
			\Tr(A_0C)+2\Tr(E_C^{1/2})
			\right\},                                  \\
			B_{\max} & :=\max_{C\in\DN}\Tr(E_C^{1/2}).
		\end{aligned}
	\]
	Then
	\[
		\E\lambda_{\max}\left(
		A_0\otimes I_M+\frac1{\sqrt M}\sum_{i=1}^n A_i\otimes G_i
		\right)
		\le
		\Psi+6 B_{\max}\sqrt{\frac{nN}{M}}
	\]
	and
	\[
		B_{\max}\le \Arow.
	\]
\end{proposition}

\begin{proof}
	By Proposition~\ref{prop:fullprocess},
	\[
		\E\lambda_{\max}(H_M)
		\le
		\E\sup_{C\in\DN}
		\left\{
		\Tr(A_0C)+\frac2{\sqrt M}\left\|\sum_{i=1}^nB_i(C)\Gamma_i\right\|_1
		\right\}.
	\]
	Introduce the block row matrix
	\[
		R_C=[B_1(C)\ B_2(C)\ \cdots\ B_n(C)]\in M_{N,nN}(\C)
	\]
	and the block Gaussian matrix
	\[
		\Gamma=\begin{bmatrix}\Gamma_1\\ \Gamma_2\\ \vdots\\ \Gamma_n\end{bmatrix}
		\in M_{nN,M}(\C),
	\]
	where $\Gamma$ is an $nN\times M$ standard complex Gaussian matrix.
	Then
	\[
		\sum_iB_i(C)\Gamma_i=R_C\Gamma,
		\qquad
		R_CR_C^*=\sum_iB_i(C)^2=E_C.
	\]
	For every $C$,
	\[
		\frac1{\sqrt M}\|R_C\Gamma\|_1
		=\Tr\left[\left(R_C\frac{\Gamma\Gamma^*}{M}R_C^*\right)^{1/2}\right].
	\]
	Since
	\[
		0\preceq R_C\frac{\Gamma\Gamma^*}{M}R_C^*
		\preceq
		\left\|\frac{\Gamma\Gamma^*}{M}\right\|_{\op}R_CR_C^*
		=
		\left\|\frac{\Gamma\Gamma^*}{M}\right\|_{\op}E_C,
	\]
	and $X\mapsto\Tr(X^{1/2})$ is monotone on the positive semidefinite cone, we get
	\[
		\frac1{\sqrt M}\|R_C\Gamma\|_1
		\le
		\left\|\frac{\Gamma\Gamma^*}{M}\right\|_{\op}^{1/2}\Tr(E_C^{1/2}).
	\]
	Therefore
	\[
		\E\lambda_{\max}(H_M)
		\le
		\E\sup_C
		\left\{
		\Tr(A_0C)+
		2\left\|\frac{\Gamma\Gamma^*}{M}\right\|_{\op}^{1/2}
		\Tr(E_C^{1/2})
		\right\}.
	\]
	For every $\alpha\ge0$, the pointwise bound
	\[
		\sup_C
		\left\{
		\Tr(A_0C)+2\alpha\Tr(E_C^{1/2})
		\right\}
		\le
		\Psi+2(\alpha-1)_+B_{\max}
	\]
	follows directly from the two cases $\alpha\le1$ and $\alpha>1$.
	Therefore
	\[
		\E\lambda_{\max}(H_M)
		\le
		\Psi+2B_{\max}\,
		\E\left[
			\left(
			\left\|\frac{\Gamma\Gamma^*}{M}\right\|_{\op}^{1/2}-1
			\right)_+
			\right].
	\]
	By Lemma~\ref{lem:wishart-estimate} with
	$d=nN$, this gives
	\[
		\E\lambda_{\max}(H_M)
		\le
		\Psi+6B_{\max}\sqrt{\frac{nN}{M}}.
	\]

	It remains to bound $B_{\max}$.  Since
	\[
		E_C^{1/2}=(R_CR_C^*)^{1/2},
	\]
	we have
	\[
		\Tr(E_C^{1/2})=\|R_C\|_1.
	\]
	Writing
	\[
		R_C=C^{1/2}[A_1C^{1/2}\ A_2C^{1/2}\ \cdots\ A_nC^{1/2}],
	\]
	the Schatten Cauchy--Schwarz inequality gives
	\[
		\|R_C\|_1
		\le
		\|C^{1/2}\|_F
		\left\|[A_1C^{1/2}\ A_2C^{1/2}\ \cdots\ A_nC^{1/2}]\right\|_F .
	\]
	As $\Tr C=1$, the first factor is one, while the square of the second is
	\[
		\Tr\left(\sum_{i=1}^n A_iCA_i\right)
		=
		\Tr\left(C\sum_{i=1}^n A_i^2\right)
		\le
		\left\|\sum_{i=1}^n A_i^2\right\|_{\op}.
	\]
	Taking the supremum over $C\in\DN$ proves
	\[
		B_{\max}\le \Arow.
	\]
\end{proof}

\upperedgeprop

\begin{proof}[Proof of Proposition~\ref{prop:upperedge}]
	By Proposition~\ref{prop:blockwishart},
	\[
		\E\lambda_{\max}(H_M)
		\le
		\Psi+6B_{\max}\sqrt{\frac{nN}{M}}.
	\]
	Proposition~\ref{prop:dual} gives $\Psi=\rho_+$.  The bound
	$B_{\max}\le\Arow$ from Proposition~\ref{prop:blockwishart} completes the
	proof.
\end{proof}

\begin{remark}[Multiplicative form when $A_0\succeq0$]
	Under the additional assumption $A_0\succeq0$, the additive error in
	Proposition~\ref{prop:upperedge} also admits the cleaner multiplicative form
	\[
		\E\lambda_{\max}(H_M)
		\le
		\left(1+3\sqrt{\frac{nN}{M}}\right)\rho_+.
	\]
	Indeed, in the proof of Proposition~\ref{prop:blockwishart}, if
	$\alpha=\|\Gamma\Gamma^*/M\|_{\op}^{1/2}$, then $A_0\succeq0$ gives the
	pointwise estimate
	\[
		\sup_C
		\left\{
		\Tr(A_0C)+2\alpha\Tr(E_C^{1/2})
		\right\}
		\le
		(1+(\alpha-1)_+)\rho_+.
	\]
	For \(\alpha>1\), this follows by writing
	\[
		\Tr(A_0C)+2\alpha\Tr(E_C^{1/2})
		=
		\alpha\{\Tr(A_0C)+2\Tr(E_C^{1/2})\}
		+(1-\alpha)\Tr(A_0C),
	\]
	and using \((1-\alpha)\Tr(A_0C)\le0\).  The case \(\alpha\le1\) is immediate.
	Moreover, \(A_0\succeq0\) implies \(\rho_+\ge0\), since
	\(\rho_+=\max_{C\in\DN}\phi(C)\) and \(\phi(C)\ge0\).
	Lemma~\ref{lem:wishart-estimate} then yields the displayed bound.  We keep the
	additive form in Proposition~\ref{prop:upperedge}, as it does not require
	$A_0\succeq0$.
\end{remark}

\section{The left edge}
\label{sec:left-edge}

We also record the corresponding lower-edge estimate.  The point is that the
Sudakov comparison itself does not require positivity of the deterministic
coefficient; the lower edge is obtained by applying the upper-edge estimate to
$-H_M$.

\leftedgeprop

\begin{proof}
	Since $G_i$ and $-G_i$ have the same
	distribution,
	\[
		-H_M
		=
		-A_0\otimes I_M-\frac1{\sqrt M}\sum_{i=1}^n A_i\otimes G_i
		\stackrel{d}{=}
		-A_0\otimes I_M+\frac1{\sqrt M}\sum_{i=1}^n A_i\otimes G_i .
	\]
	Therefore Proposition~\ref{prop:blockwishart}, applied with deterministic
	coefficient $-A_0$, gives
	\[
		\E\lambda_{\max}(-H_M)
		\le
		\max_{C\in\DN}
		\left\{
		-\Tr(A_0C)+2\Tr(E_C^{1/2})
		\right\}
		+
		6\sqrt{\frac{nN}{M}}\,\Arow.
	\]
	By Proposition~\ref{prop:leftdual}, the maximum above is \(-\rho_-\).  Using
	$\lambda_{\max}(-H_M)=-\lambda_{\min}(H_M)$, this is equivalent to
	\[
		\E\lambda_{\min}(H_M)
		\ge
		\rho_-
		-
		6\sqrt{\frac{nN}{M}}\,\Arow.
	\]
\end{proof}

\section{Lipschitz control and proof of the main theorem}
\label{sec:op-norm}

\begin{lemma}[Lipschitz control of spectral observables]
	\label{lem:spectral-lipschitz}
	On the real vector space
	\(\bigl(M_M(\C)_{\mathrm{sa}}\bigr)^n\), the maps
	\[
		(G_1,\ldots,G_n)\mapsto\lambda_{\max}(H_M),
		\qquad
		(G_1,\ldots,G_n)\mapsto\lambda_{\max}(-H_M),
		\qquad
		(G_1,\ldots,G_n)\mapsto\|H_M\|_{\op}
	\]
	are Lipschitz with respect to the real Euclidean norm on the GUE inputs induced
	by the Frobenius norm,
	\[
		\|(\Delta_1,\ldots,\Delta_n)\|_{\mathrm{Euc}}
		=
		\left(\sum_{i=1}^n\|\Delta_i\|_F^2\right)^{1/2},
	\]
	with constant
	\[
		L
		\le
		\frac{\Arow}{\sqrt M}.
	\]
\end{lemma}

\begin{proof}
	If $G_i$ is replaced by $G_i+\Delta_i$, with each $\Delta_i$ Hermitian, then
	the corresponding perturbation of $H_M$ is
	\[
		\Delta H=
		\frac1{\sqrt M}\sum_{i=1}^n A_i\otimes\Delta_i.
	\]
	By the variational characterization of $\lambda_{\max}$ and by the triangle
	inequality for the operator norm, all three observables in the statement
	change by at most $\|\Delta H\|_{\op}$.  We bound this operator norm using the
	Frobenius norm of the perturbations.  Since $\Delta H$ is Hermitian, it is enough
	to test Rayleigh quotients.  Let $S\in M_{N,M}(\C)$ with $\|S\|_F=1$.  Under the
	vectorization convention from Section~\ref{sec:setup},
	\[
		\begin{aligned}
			\left|
			\ip{\vecop(S)}{\Delta H\vecop(S)}
			\right|
			 & \le
			\frac1{\sqrt M}
			\sum_{i=1}^n
			\left|\Tr(S^*A_iS\Delta_i^{\top})\right| \\
			 & \le
			\frac1{\sqrt M}
			\left(\sum_{i=1}^n\|\Delta_i\|_F^2\right)^{1/2}
			\left(\sum_{i=1}^n\|S^*A_iS\|_F^2\right)^{1/2}.
		\end{aligned}
	\]
	By the ideal property of the Frobenius norm and since
	$\|S\|_{\op}\le\|S\|_F=1$,
	\[
		\|S^*A_iS\|_F\le\|S\|_{\op}\|A_iS\|_F \leq \|A_iS\|_F .
	\]
	Thus
	\[
		\sum_{i=1}^n\|S^*A_iS\|_F^2
		\le
		\sum_{i=1}^n\|A_iS\|_F^2
		=
		\Tr S^*\left(\sum_{i=1}^n A_i^2\right)S
		\le
		\left\|\sum_{i=1}^n A_i^2\right\|_{\op}.
	\]
	Taking the supremum over such $S$ gives
	\[
		\|\Delta H\|_{\op}
		\le
		\frac{\Arow}{\sqrt M}
		\left(\sum_{i=1}^n\|\Delta_i\|_F^2\right)^{1/2},
	\]
	which proves the claim.
\end{proof}

\mainthm

\begin{proof}[Proof of Theorem~\ref{thm:main}]
	Write
	\[
		X=\lambda_{\max}(H_M),
		\qquad
		Y=-\lambda_{\min}(H_M)=\lambda_{\max}(-H_M).
	\]
	Then
	\[
		\|H_M\|_{\op}=\max\{X,Y\}.
	\]
	Let
	\[
		\mu=
		\max\{\rho_+,-\rho_-\}
		+
		6\sqrt{\frac{nN}{M}}\,\Arow.
	\]
	Then, Proposition~\ref{prop:upperedge} and
	Proposition~\ref{prop:leftedge} give $\E X\le\mu$ and $\E Y\le\mu$.

	We now use Gaussian concentration in the real Euclidean coordinates
	corresponding to the Frobenius norm and our GUE normalization.
	By Lemma~\ref{lem:spectral-lipschitz} and the Gaussian concentration
	inequality for Lipschitz functions
	\cite[Theorem~5.6]{BoucheronLugosiMassart2013}, the centered random variables
	satisfy, for every $t\ge0$,
	\[
		\mathbb P\{X-\E X\ge t\}\le e^{-t^2/(2L^2)},
		\qquad
		\mathbb P\{Y-\E Y\ge t\}\le e^{-t^2/(2L^2)},
	\]
	where
	\[
		L
		\le
		\frac{\Arow}{\sqrt M}.
	\]
	Since $\E X\le \mu$ and $\E Y\le \mu$, we have
	\[
		\|H_M\|_{\op}
		=
		\max\{X,Y\}
		\le
		\mu+\max\{X-\E X,\,Y-\E Y\}.
	\]
	Therefore, by the standard maximal inequality for two subgaussian random
	variables, or directly by the union bound and tail integration with
	\(\min\{1,2e^{-t^2/(2L^2)}\}\le2e^{-t^2/(2L^2)}\),
	\[
		\begin{aligned}
			\E\max\{X-\E X,\,Y-\E Y\}
			 & \le
			\E\bigl(\max\{X-\E X,\,Y-\E Y\}\bigr)_+ \\
			 & \le
			2\int_0^\infty e^{-t^2/(2L^2)}\,dt      \\
			 & =
			\sqrt{2\pi}\,L
			\le
			3L .
		\end{aligned}
	\]
	Thus
	\[
		\E\|H_M\|_{\op}\le \mu+3L.
	\]
	The Lipschitz correction is smaller than the block-Wishart error term and is not
	the bottleneck in the estimate.  Indeed,
	\[
		L
		\le
		\frac{\Arow}{\sqrt M}
		\le
		\sqrt{\frac{nN}{M}}\,\Arow,
	\]
	so
	\[
		\E\|H_M\|_{\op}
		\le
		\rho_*
		+
		9\sqrt{\frac{nN}{M}}\,\Arow.
	\]
\end{proof}

\begin{remark}[The finite-dimensional constant]
	Before the final simplification, the proof gives the sharper estimate
	\[
		\E\|H_M\|_{\op}
		\le
		\rho_*+
		6\sqrt{\frac{nN}{M}}\,\Arow
		+
		\frac{3}{\sqrt M}\,\Arow.
	\]
	The numerical constants have not been optimized.  For example, retaining
	\(1+\sqrt{\pi/2}\) in Lemma~\ref{lem:wishart-estimate}, rather than rounding
	it to \(3\), already lowers the single combined constant \(9\) to about
	\(7.01\).  The estimate is aimed principally at regimes in which \(N\)
	grows with \(M\); the constant \(9\) is not intended to describe fine
	finite-size behavior.
\end{remark}

\section{Bounds from positive semidefinite Kraus decompositions}
\label{sec:rekraus}

The positivity hypothesis in Theorem~\ref{thm:main} is not intrinsic to the
displayed family $(A_i)_{i=1}^n$.  Indeed, both the law of the random model and
the Lehner edges only see the covariance map \(\eta\) in
\eqref{eq:covariance-map}.  This observation gives a practical way to apply the
preceding results to genuinely indefinite presentations of models that also
admit a positive presentation.

\begin{definition}[Positive semidefinite Kraus decomposition]
	\label{def:psd-rekraus}
	We say that the covariance map \(\eta\) \emph{admits a positive
		semidefinite Kraus decomposition} if there are positive semidefinite matrices
	\[
		\widetilde A_1,\ldots,\widetilde A_r\succeq0
	\]
	such that
	\begin{equation}
		\label{eq:psd-rekraus}
		\eta(X)=\sum_{j=1}^r\widetilde A_jX\widetilde A_j
		\qquad\text{for every }X\in M_N(\C).
	\end{equation}
\end{definition}

Thus the assumption is on the covariance map generated by the matrices, rather
than on the signs of a particular choice of generators.  Equivalently, after
adjoining zero matrices if necessary, one can pass from $(A_i)$ to a positive
semidefinite family by a real orthogonal mixing.  We record the invariance and
this equivalence explicitly.

\begin{lemma}[Orthogonal mixing invariance]
	\label{lem:orthogonal-mixing}
	Let $A_1,\ldots,A_k$ be Hermitian, with some of them possibly zero, let
	$O\in O(k)$, and put
	\[
		\widetilde A_j=\sum_{i=1}^k O_{ji}A_i.
	\]
	If $G_1,\ldots,G_k$ are independent GUE matrices, then
	\[
		A_0\otimes I_M+\frac1{\sqrt M}\sum_{i=1}^k A_i\otimes G_i
		\ \stackrel{d}{=}\
		A_0\otimes I_M+\frac1{\sqrt M}\sum_{j=1}^k
		\widetilde A_j\otimes G_j.
	\]
\end{lemma}

\begin{proof}
	Write $\widehat G_i=\sum_jO_{ji}G_j$.  The tuple
	$(\widehat G_i)_{i=1}^k$ is centered jointly Gaussian and has the same
	covariance as $(G_i)_{i=1}^k$, hence the two tuples have the same law.
	Moreover,
	\[
		\sum_j\widetilde A_j\otimes G_j
		=\sum_iA_i\otimes\widehat G_i,
	\]
	which proves the assertion.
\end{proof}

\begin{proposition}[$\eta$-equivalence of Hermitian families]
	\label{prop:eta-equivalence}
	Two finite Hermitian families $(A_i)$ and $(\widetilde A_j)$ generate the
	same map in \eqref{eq:covariance-map} if and only if the following holds.
	View each family as a vector of matrices and append zero matrices to the
	shorter vector so that the two vectors have the same length, say $k$.
	Then there is a real orthogonal matrix $O\in O(k)$ such that, after indexing
	the padded families by $1,\ldots,k$,
	\[
		\widetilde A_j=\sum_{i=1}^k O_{ji}A_i,
		\qquad 1\leq j\leq k.
	\]
\end{proposition}

\begin{proof}
	Suppose first that the two padded families are related by the orthogonal
	mixing displayed in the statement.  Then, for every $X\in M_N(\C)$,
	\[
		\sum_{j=1}^k\widetilde A_jX\widetilde A_j
		=\sum_{i,\ell=1}^k
		\left(\sum_{j=1}^kO_{ji}O_{j\ell}\right)A_iXA_\ell
		=\sum_{i=1}^kA_iXA_i,
	\]
	where the last equality follows from $O^{\mathsf T}O=I_k$.  Thus the two
	families generate the same map $\eta$.  Conversely, equality of the maps
	says that $(A_i)$ and $(\widetilde A_j)$ are two Kraus decompositions of the
	same completely positive map.  First append zeros to the shorter family so
	that both have the common length \(k\); this does not change either map.
	Use real orthogonal changes of generators in \(O(k)\) to remove real linear
	dependencies, recording each removed direction as a zero matrix.  For
	Hermitian matrices, real and complex linear independence are equivalent:
	a complex relation and its adjoint imply separately that the real and
	imaginary parts of the coefficients give real relations.  The resulting
	nonzero families are therefore minimal Kraus representations, of the same
	length, say \(q\).
	The unitary freedom of minimal Kraus decompositions
	\cite[Corollary~2.23]{Watrous2018} gives a unitary matrix $U=(u_{ji})$ such
	that
	$\widetilde A_j=\sum_i u_{ji}A_i$.  Since both sides are Hermitian,
	\[
		\sum_i\Ima(u_{ji})A_i=0.
	\]
	Real linear independence of the nonzero $A_i$ forces all the relevant
	coefficients $u_{ji}$ to be real.  Thus $U\in O(q)$.  Extend it to
	\(U\oplus I_{k-q}\in O(k)\), undo the initial real orthogonal changes, and
	compose the three orthogonal matrices.  This produces the matrix
	\(O\in O(k)\) in the statement.
\end{proof}

It follows that the law of $H_M$ depends on the coefficient family only through
$(A_0,\eta)$.  The same is true of all three free quantities because
\[
	A_0+Z+\sum_iA_iZ^{-1}A_i=A_0+Z+\eta(Z^{-1}).
\]
Moreover,
\begin{equation}
	\label{eq:eta-identity}
	\eta(I_N)=\sum_iA_i^2=\sum_j\widetilde A_j^2.
\end{equation}
We may therefore apply the positive-coefficient estimates to any positive
semidefinite Kraus family representing $\eta$.

\psdrekrausthm

\begin{proof}[Proof of Theorem~\ref{thm:psd-rekraus}]
	Proposition~\ref{prop:eta-equivalence} identifies the two coefficient
	families, after padding both to the common length \(\max\{n,r\}\), by an
	orthogonal mixing.
	Lemma~\ref{lem:orthogonal-mixing} therefore shows that their random models
	have the same law.  Their Lehner edges are also the same, and the
	square-function term of the positive semidefinite family is
	$\|\eta(I_N)\|_{\op}^{1/2}$ by \eqref{eq:eta-identity}.  Apply
	Propositions~\ref{prop:upperedge} and~\ref{prop:leftedge}, and
	Theorem~\ref{thm:main}, after discarding the padded zero coefficients from
	the positive family.  Thus its effective length is \(r\), rather than
	\(\max\{n,r\}\).
\end{proof}

Consequently, the asymptotic conclusion of Corollary~\ref{cor:asymptotic}
continues to hold whenever $r$ stays bounded, $N=o(M)$, and
$\|\eta(I_N)\|_{\op}$ stays bounded.  The following examples show concretely
that the criterion applies to indefinite coefficient families.  Write
\[
	\sigma_x=\begin{pmatrix}0&1\\1&0\end{pmatrix},
	\qquad
	\sigma_z=\begin{pmatrix}1&0\\0&-1\end{pmatrix}.
\]

\begin{example}[An indefinite coefficient removed by rotation]
	\label{ex:pauli-pair}
	Consider the family
	\[
		(A_1,A_2)=(\sigma_x,I_2).
	\]
	The matrix $\sigma_x$ is indefinite.  Nevertheless, the orthogonal change of
	generators
	\[
		\widetilde A_1=\frac{I_2+\sigma_x}{\sqrt2},
		\qquad
		\widetilde A_2=\frac{I_2-\sigma_x}{\sqrt2}
	\]
	produces two positive semidefinite matrices.  Hence
	Theorem~\ref{thm:psd-rekraus} applies with $r=2$, with exactly the same error
	as in the positive two-coefficient case.
\end{example}

\begin{proposition}[A nontrivial Pauli family]
	\label{prop:pauli-threshold}
	For $t>0$, consider the signed family
	\[
		\mathfrak F_t=(\sigma_x,\sigma_z,tI_2).
	\]
	Its covariance map admits a positive semidefinite Kraus decomposition if and
	only if $t\ge\sqrt2$.  When $t\ge\sqrt2$, one may use the four matrices
	\[
		\widetilde A_{\varepsilon_1,\varepsilon_2}
		=
		\frac12\bigl(\varepsilon_1\sigma_x+\varepsilon_2\sigma_z\bigr)
		+\frac t2I_2,
		\qquad
		(\varepsilon_1,\varepsilon_2)\in\{\pm1\}^2.
	\]
	At $t=\sqrt2$, these are rank-one positive semidefinite matrices with
	eigenvalues $0$ and $\sqrt2$.
\end{proposition}

\begin{proof}
	Order the elements of $\{\pm1\}^2$ as
	\[
		(+,+),(+,-),(-,+),(-,-).
	\]
	Then
	\[
		\begin{pmatrix}
			\widetilde A_{+,+} \\
			\widetilde A_{+,-} \\
			\widetilde A_{-,+} \\
			\widetilde A_{-,-}
		\end{pmatrix}
		=
		\frac12
		\begin{pmatrix}
			1  & 1  & 1 & 1  \\
			1  & -1 & 1 & -1 \\
			-1 & 1  & 1 & -1 \\
			-1 & -1 & 1 & 1
		\end{pmatrix}
		\begin{pmatrix}
			\sigma_x \\
			\sigma_z \\
			tI_2     \\
			0
		\end{pmatrix}.
	\]
	The displayed coefficient matrix is orthogonal.  Hence the family
	$(\widetilde A_{\varepsilon_1,\varepsilon_2})$ is an orthogonal mixing of
	$(\sigma_x,\sigma_z,tI_2,0)$, and therefore has the same covariance map.  Its
	two eigenvalues are
	$(t\pm\sqrt2)/2$, so it is positive semidefinite exactly when
	$t\ge\sqrt2$.

	For necessity, let $(\widetilde A_j)$ be any positive semidefinite Kraus
	family for this map.  Proposition~\ref{prop:eta-equivalence} gives a real
	isometry $(o_{jk})$ such that
	\[
		\widetilde A_j
		=o_{j1}\sigma_x+o_{j2}\sigma_z+t\,o_{j3}I_2.
	\]
	Positivity forces
	$t\,o_{j3}\ge(o_{j1}^2+o_{j2}^2)^{1/2}$.  Squaring and summing in $j$, and
	using orthonormality of the three columns, gives $t^2\ge2$.
\end{proof}

\begin{remark}
	For a displayed coefficient tuple, the condition in
	Definition~\ref{def:psd-rekraus} is genuinely weaker than termwise
	positivity, but it is not automatic for a signed family.  It does not,
	however, enlarge the class of random-matrix laws represented by positive
	families.  An explicit orthogonal mixing into positive semidefinite matrices
	is a directly verifiable certificate.  The next section explains what fails
	in the direct comparison when no such certificate is available.
\end{remark}

\section{Failure of the direct comparison for signed coefficients}
\label{sec:signed}

In this final section, we delimit the scope of the particular increment
comparison used in Section~\ref{sec:full-process}.  The occurrence of negative
eigenvalues in one Kraus family is not itself an obstruction, as
Section~\ref{sec:rekraus} shows, but a general covariance map need not admit a
positive semidefinite Kraus decomposition.  What follows is an obstruction to
our sufficient comparison condition, not to the desired inequality itself.

Let $A_i=A_i^*$ and $B_i(C)=C^{1/2}A_iC^{1/2}$.  Then $B_i(C)=B_i(C)^*$ and
\[
	B_i(C)^2=C^{1/2}A_iCA_iC^{1/2}.
\]
Consequently,
\[
	\sum_{i=1}^nB_i(C)^2
	=C^{1/2}\left(\sum_{i=1}^nA_iCA_i\right)C^{1/2}.
\]
Thus the fixed-fiber expression
\[
	\Tr(A_0C)+2\Tr\left[\left(C^{1/2}\eta(C)C^{1/2}\right)^{1/2}\right]
\]
and the variational formula over $C\in\DN$ remain valid without positivity of the
individual $A_i$.  Here a \emph{fiber} means the subset of the process index
set \(\DN\times\St_{N,M}\) on which \(C\) is fixed and only \(W\) varies.

\begin{proposition}[Failure of the signed fixed-fiber increment bound]
	\label{prop:signed-increment-failure}
	For signed Hermitian \(B\), the inequality
	\[
		\|W^*BW-V^*BV\|_F^2\le2\|B(W-V)\|_F^2
	\]
	need not hold, even when \(N=M=2\).
\end{proposition}

\begin{proof}
	Take
	\[
		W=I_2,
		\qquad
		B=\begin{pmatrix}1&0\\0&-1\end{pmatrix},
		\qquad
		V=\begin{pmatrix}\cos\theta & -\sin\theta \\
               \sin\theta & \cos\theta\end{pmatrix}.
	\]
	A direct computation gives
	\[
		\|B-V^*BV\|_F^2=8\sin^2\theta,
		\qquad
		2\|B(I-V)\|_F^2=8(1-\cos\theta).
	\]
	The desired inequality would require
	\(\sin^2\theta\le1-\cos\theta\), which fails for small nonzero
	\(\theta\).
\end{proof}

For a fixed Hermitian $B$ one still has the elementary estimate
\[
	W^*BW-V^*BV=W^*B(W-V)+(W-V)^*BV,
\]
and hence
\[
	\|W^*BW-V^*BV\|_F\le2\|B(W-V)\|_F.
\]
After squaring, this gives the weaker bound
\[
	\|W^*BW-V^*BV\|_F^2\le4\|B(W-V)\|_F^2.
\]
At a fixed fiber, the Sudakov argument then yields a comparison with a Gaussian
process whose amplitude is larger by a factor $\sqrt2$.  Thus the deterministic
comparison functional
\[
	\Tr(A_0C)+2\Tr(E_C^{1/2})
\]
would be replaced by
\[
	\Tr(A_0C)+2\sqrt2\,\Tr(E_C^{1/2}),
\]
which no longer recovers Lehner's leading constant.  This observation is only a
fixed-fiber calculation: extending it to the full process would require an
increment estimate for two different matrices $B_i(C)$ and $B_i(C')$, which is
precisely where the positive-coefficient proof uses the Araki--Yamagami
inequality.

To conclude, Proposition~\ref{prop:signed-increment-failure} shows only that the
increment hypothesis underlying our direct Sudakov--Fernique comparison fails
for general signed families.  It neither disproves the desired inequality nor
supplies a finite-dimensional bound for such families.  Indeed, the general
signed asymptotic statement follows by other methods from
\cite{CGVvHClassical}.  Recovering the free leading constant for every
Hermitian family appears to require a different comparison or a different
method.

\section*{Acknowledgements}

The minimax duality of Proposition~\ref{prop:dual}---equating Lehner's infimum
over positive definite matrices with the supremum over density matrices via
Sion's theorem---was pointed out to the first author by R.~van Handel during a
visit to Princeton University.  We are deeply grateful to him for this insight,
for many thoughtful exchanges on the scope and presentation of the result, and
for a very careful reading of an earlier version of this paper.
The first author also thanks G.~Aubrun and C.~Bordenave for useful discussions on Sudakov--Fernique comparison techniques at an earlier stage of this project.
The authors also acknowledge ChatGPT 5.5 for suggesting the use of
\cite{ArakiYamagami1981} in the comparison process.

The authors were supported by JSPS Grant-in-Aid Scientific Research (A)
no.~25H00593, and Challenging Research (Exploratory) no.~23K17299.

\bibliographystyle{plain}
\bibliography{references}

@article{Lehner,
  author = {Lehner, Franz},
  title = {Computing norms of free operators with matrix coefficients},
  journal = {American Journal of Mathematics},
  volume = {121},
  number = {3},
  pages = {453--486},
  year = {1999},
  doi = {10.1353/ajm.1999.0022},
}

@article{KuniskyLehnerSDP2026,
  author = {Kunisky, Dmitriy},
  title = {Lehner's operator norm formulas, semidefinite programming, and spiked
           matrix models},
  journal = {arXiv preprint arXiv:2606.14687},
  year = {2026},
  eprint = {2606.14687},
  archivePrefix = {arXiv},
}

@article{ParmaksizVanHandel2025,
  author = {Parmaksiz, Emre and van Handel, Ramon},
  title = {Computing extreme singular values of free operators},
  journal = {arXiv preprint arXiv:2510.23987},
  year = {2025},
  eprint = {2510.23987},
  archivePrefix = {arXiv},
}

@article{HaagerupThorbjornsen,
  author = {Haagerup, Uffe and Thorbj{\o}rnsen, Steen},
  title = {A new application of random matrices: {$\mathrm{Ext}(C^*_{\mathrm{red}}(F_2))$} is not a group},
  journal = {Annals of Mathematics},
  series = {2},
  volume = {162},
  number = {2},
  pages = {711--775},
  year = {2005},
}

@article{CollinsParraudGuionnet,
  author = {Collins, Beno{\^i}t and Guionnet, Alice and Parraud, Felix},
  title = {On the operator norm of non-commutative polynomials in deterministic
           matrices and iid {GUE} matrices},
  journal = {Cambridge Journal of Mathematics},
  volume = {10},
  number = {1},
  pages = {195--260},
  year = {2022},
  doi = {10.4310/cjm.2022.v10.n1.a3},
}

@article{Anderson2013,
  author = {Anderson, Greg W.},
  title = {Convergence of the largest singular value of a polynomial in
           independent {W}igner matrices},
  journal = {The Annals of Probability},
  volume = {41},
  number = {3B},
  pages = {2103--2181},
  year = {2013},
}

@article{BBvH2023,
  author = {Bandeira, Afonso S. and Boedihardjo, March T. and van Handel, Ramon},
  title = {Matrix concentration inequalities and free probability},
  journal = {Inventiones Mathematicae},
  volume = {234},
  number = {1},
  pages = {419--487},
  year = {2023},
  doi = {10.1007/s00222-023-01204-6},
}

@article{BCSvH2026,
  author = {Bandeira, Afonso S. and Cipolloni, Giorgio and Schr{\"o}der,
            Dominik and van Handel, Ramon},
  title = {Matrix concentration inequalities and free probability {II}.
           {Two-sided} bounds and applications},
  journal = {Communications of the American Mathematical Society},
  volume = {6},
  pages = {896--946},
  year = {2026},
  doi = {10.1090/cams/75},
  eprint = {2406.11453},
  archivePrefix = {arXiv},
}

@article{vanHandel2017SpectralNorm,
  author = {van Handel, Ramon},
  title = {On the spectral norm of {G}aussian random matrices},
  journal = {Transactions of the American Mathematical Society},
  volume = {369},
  number = {11},
  pages = {8161--8178},
  year = {2017},
  doi = {10.1090/tran/6922},
  eprint = {1502.05003},
  archivePrefix = {arXiv},
}

@article{BuchholzKhintchine,
  author = {Buchholz, Artur},
  title = {Operator {Khintchine} inequality in non-commutative probability},
  journal = {Mathematische Annalen},
  volume = {319},
  number = {1},
  pages = {1--16},
  year = {2001},
  doi = {10.1007/PL00004425},
}

@book{PisierOperatorSpace,
  author = {Pisier, Gilles},
  title = {Introduction to Operator Space Theory},
  publisher = {Cambridge University Press},
  address = {Cambridge},
  year = {2003},
}

@article{PisierSubexponentialOperatorSpaces,
  author = {Pisier, Gilles},
  title = {Random matrices and subexponential operator spaces},
  journal = {Israel Journal of Mathematics},
  volume = {203},
  pages = {223--273},
  year = {2014},
  doi = {10.1007/s11856-014-1069-0},
  eprint = {1212.2053},
  archivePrefix = {arXiv},
}

@article{CGVvHClassical,
  author = {Chen, Chi-Fang and Garza-Vargas, Jorge and van Handel, Ramon},
  title = {A new approach to strong convergence {II}. {The} classical ensembles},
  journal = {Geometric and Functional Analysis},
  year = {2026},
  note = {To appear},
  doi = {10.1007/s00039-026-00744-2},
  eprint = {2412.00593},
  archivePrefix = {arXiv},
}

@article{CGVTvHClassicalI,
  author = {Chen, Chi-Fang and Garza-Vargas, Jorge and Tropp, Joel A. and van
            Handel, Ramon},
  title = {A new approach to strong convergence},
  journal = {Annals of Mathematics},
  series = {2},
  volume = {203},
  number = {2},
  pages = {555--602},
  year = {2026},
  doi = {10.4007/annals.2026.203.2.4},
}

@article{ParraudTensor2024,
  author = {Parraud, F{\'e}lix},
  title = {The spectrum of a tensor of random and deterministic matrices},
  journal = {arXiv preprint arXiv:2410.04481},
  year = {2024},
  eprint = {2410.04481},
  archivePrefix = {arXiv},
}

@incollection{vanHandelSurvey2025,
  author = {van Handel, Ramon},
  title = {The strong convergence phenomenon},
  booktitle = {Current Developments in Mathematics 2025},
  publisher = {International Press},
  pages = {177--261},
  year = {2026},
  eprint = {2507.00346},
  archivePrefix = {arXiv},
}

@article{Voiculescu1991,
  author = {Voiculescu, Dan},
  title = {Limit laws for random matrices and free products},
  journal = {Inventiones Mathematicae},
  volume = {104},
  number = {1},
  pages = {201--220},
  year = {1991},
  doi = {10.1007/BF01245072},
}

@article{ArakiYamagami1981,
  author = {Araki, Huzihiro and Yamagami, Shigeru},
  title = {An inequality for {Hilbert-Schmidt} norm},
  journal = {Communications in Mathematical Physics},
  volume = {81},
  number = {1},
  pages = {89--96},
  year = {1981},
  doi = {10.1007/BF01941801},
}

@incollection{DavidsonSzarek2001,
  author = {Davidson, Kenneth R. and Szarek, Stanis{\l}aw J.},
  title = {Local operator theory, random matrices and {B}anach spaces},
  booktitle = {Handbook of the Geometry of Banach Spaces, Vol. I},
  pages = {317--366},
  publisher = {North-Holland},
  address = {Amsterdam},
  year = {2001},
}

@incollection{LedouxLargestEigenvalues,
  author = {Ledoux, Michel},
  title = {A remark on hypercontractivity and tail inequalities for the largest
           eigenvalues of random matrices},
  booktitle = {S{\'e}minaire de Probabilit{\'e}s XXXVII},
  series = {Lecture Notes in Mathematics},
  volume = {1832},
  pages = {360--369},
  publisher = {Springer},
  address = {Berlin},
  year = {2003},
}

@incollection{AubrunLargestEigenvalue,
  author = {Aubrun, Guillaume},
  title = {A sharp small deviation inequality for the largest eigenvalue of a
           random matrix},
  booktitle = {S{\'e}minaire de Probabilit{\'e}s XXXVIII},
  series = {Lecture Notes in Mathematics},
  volume = {1857},
  pages = {320--337},
  publisher = {Springer},
  address = {Berlin},
  year = {2005},
}

@book{Bhatia1997,
  author = {Bhatia, Rajendra},
  title = {Matrix Analysis},
  series = {Graduate Texts in Mathematics},
  volume = {169},
  publisher = {Springer},
  address = {New York},
  year = {1997},
}

@book{TomczakJaegermann1989,
  author = {Tomczak-Jaegermann, Nicole},
  title = {Banach--Mazur Distances and Finite-Dimensional Operator Ideals},
  series = {Pitman Monographs and Surveys in Pure and Applied Mathematics},
  volume = {38},
  publisher = {Longman Scientific \& Technical},
  address = {Harlow},
  year = {1989},
}

@article{Sion,
  author = {Sion, Maurice},
  title = {On general minimax theorems},
  journal = {Pacific Journal of Mathematics},
  volume = {8},
  number = {1},
  pages = {171--176},
  year = {1958},
  doi = {10.2140/pjm.1958.8.171},
}

@book{Vershynin,
  author = {Vershynin, Roman},
  title = {High-Dimensional Probability: An Introduction with Applications in
           Data Science},
  publisher = {Cambridge University Press},
  address = {Cambridge},
  year = {2018},
}

@book{BoucheronLugosiMassart2013,
  author = {Boucheron, St{\'e}phane and Lugosi, G{\'a}bor and Massart, Pascal},
  title = {Concentration Inequalities: A Nonasymptotic Theory of Independence},
  publisher = {Oxford University Press},
  address = {Oxford},
  year = {2013},
}

@article{Vitale,
  author = {Vitale, Richard A.},
  title = {Some comparisons for {Gaussian} processes},
  journal = {Proceedings of the American Mathematical Society},
  volume = {128},
  number = {10},
  pages = {3043--3046},
  year = {2000},
  doi = {10.1090/S0002-9939-00-05367-3},
}

@book{Watrous2018,
  author = {Watrous, John},
  title = {The Theory of Quantum Information},
  publisher = {Cambridge University Press},
  year = {2018},
  isbn = {9781107180567},
}

@article{Stojnic2026,
  author = {Stojnic, Mihailo},
  title = {An {RDT} based confirmation of {Lehner}'s formula for
           {Kronecker}--{Gaussian} matrices},
  journal = {arXiv preprint arXiv:2607.26551},
  year = {2026},
  eprint = {2607.26551},
  archivePrefix = {arXiv},
}

\end{document}